\documentclass[12pt,a4paper]{article}
\usepackage[utf8]{inputenc}
\usepackage[T1]{fontenc}
\usepackage{authblk}

\usepackage{amsmath, amscd, amsthm, amscd, amsfonts, amssymb, graphicx, color, soul, enumerate, xcolor,  mathrsfs, latexsym, bigstrut, framed, caption}
\usepackage{pstricks,pst-node,pst-tree}
\usepackage[bookmarksnumbered, colorlinks, 
plainpages,backref]{hyperref}
\usepackage{blindtext}
\hypersetup{colorlinks=true,linkcolor=blue, anchorcolor=green, citecolor=midblue, urlcolor=darkblue, filecolor=magenta, pdftoolbar=true}
\usepackage[all]{xy}

\usepackage{tikz}
\usetikzlibrary{calc,decorations.pathreplacing,decorations.markings,positioning,shapes}

\definecolor{cupgreen}{rgb}{0,0.498,0.208}
\definecolor{cupblue}{rgb}{0,0,.5}
\definecolor{cupred}{rgb}{1,0.04,0}
\definecolor{cuppink}{rgb}{0.925,0,0.545}
\definecolor{cupmagenta}{rgb}{0.624,0.161,0.424}
\definecolor{cupbrown}{rgb}{0.71,0.212,0.133}

\definecolor{cupgreen}{rgb}{0,0,0}
\definecolor{cupblue}{rgb}{0,0,0}
\definecolor{cupred}{rgb}{0,0,0}
\definecolor{cuppink}{rgb}{0,0,0}
\definecolor{cupmagenta}{rgb}{0,0,0}
\definecolor{cupbrown}{rgb}{0,0,0}

\definecolor{TITLE}{rgb}{0,0,0}
\definecolor{midblue}{rgb}{0.00,0.0,0.80}
\definecolor{darkblue}{rgb}{0.00,0.00,0.45}
\definecolor{SECTION}{rgb}{0.50,0.00,1.00}
\definecolor{THM}{rgb}{0.8,0,0.1}
\definecolor{SEC}{rgb}{0,0,1}

\newcommand{\aut}{\mathrm{Aut}}

\makeatother
\normalsize
\newtheorem{theorem}{{\color{THM} Theorem}}[section]

\newtheorem{lemma}[theorem]{{\color{THM}Lemma}}

\newtheorem{corollary}[theorem]{{\color{THM}Corollary}}

\theoremstyle{definition}
\newtheorem{definition}[theorem]{{\color{THM}Definition\ }}

\newtheorem{example}[theorem]{{\color{THM}Example}}

\numberwithin{equation}{section}

\date{}
\title{Distinguishing Polynomials of Graphs}

\author[1]{M. H. Shirdareh Haghighi\thanks{shirdareh@shirazu.ac.ir}}

\author[1]{A. M. Ghazanfari\thanks{amir.m.ghazanfari@gmail.com}}

\author[1]{S. A. Talebpour Shirazi Fard\thanks{seyed.alireza.talebpour@gmail.com}}

\affil[1]{Department of Mathematics, Shiraz University, Shiraz, Iran}

\begin{document}
	\maketitle



	\begin{abstract}
	For a graph $G$,  a $k$-coloring $c:V(G)\to \{1,2,\ldots, k\}$ is called distinguishing, if the only automorphism $f$ of $G$ with the property $c(v)=c(f(v))$ for every vertex $v\in G$ (color-preserving automorphism), is the identity. In this paper, we show that 
the number of distinguishing $k$-colorings of $G$ is a monic polynomial in $k$, calling it the distinguishing polynomial of $G$. Furthermore, we compute the distinguishing polynomials of cycles and complete multipartite graphs. We also show that the multiplicity of zero as a root of the distinguishing polynomial of $G$ is at least the number of orbits of $G$. \\

	
	\end{abstract}

\section{Introduction, The Distinguishing Polynomial}

Let $G$ be a simple graph.  A $k$-coloring $c:V(G) \to \{1,2,\ldots, k\}$  is called distinguishing, if the only automorphism of $G$ with the property $c(v)=c(f(v))$ for every vertex $v\in G$  is the identity.
In other words,   the vertices of $G$ are colored such that each non-identity automorphism of $G$ changes the color of some vertex. Equivalently, sometimes we say, a distinguishing coloring breaks all symmetries of $G.$ In 1996, the pioneer work of Albertson \cite{albertson}
is published and after that this notion is studied and extended by many authors.  


The distinguishing number of $G$, denoted by $D(G)$, is the minimum $k$ for which a distinguishing $k$-coloring exists. Since the automorphisms of the graph $G^c$ are the same as those of $G$, we have $D(G) = D(G^c)$. 

In \cite{kwz} the upper bound $\Delta(G)+1$ is given for $D(G)$, with equlity only for complete graphs, balanced complete multipartite graphs $K_{n,n}$ and $C_5$. Distinguishing graphs by edge colorings and total colorings are  introduced in \cite{edge} and \cite{total}.  Infinite graphs are also studied \cite{imrich}. Our main reference is \cite{ahmadi2020number} where nonequivalent distinguishing colorings is defined.




Let $c: V(G) \to \{1,2,\ldots, k\}$ be a coloring (not necessarily distinguishing) of $G$. For each automorphism $f$  of $G$, we have a permutation of the vertices which, at the same time, moves the colors.  So we can define a $k$-coloring $c'$ by 
\begin{equation}\label{equivalent colorings}
c'(v)=c(f^{-1}(v)	\;\;\; ( \mbox{i.e.} \;\;\; c'(f(v))=c(v)).
\end{equation}
These two colorings $c$ and $c'$ are called equivalent. By the same token, we call two colorings $c$ and $c'$ equivalent if there exists an automorphism $f$ of $G$ which establishes equation \ref{equivalent colorings}. Otherwise, they are non-equivalent. Plainly, if $c$ and $c'$ are equivalent, then $c$ is distinguishing if and only if $c'$ is distinguishing.

Non-equivalent distinguishing $k$-colorings is first defined in \cite{ahmadi2020number}.  They denote by  $\Phi_k(G)$  the number of non-equivalent $k$-colorings of $G$;  and by $\phi_k(G)$ the number of non-equivalent $k$-colorings of $G$ in which all $k$ colors are used. So, we have:
\begin{equation}\label{phi-phi}
	\Phi_k(G) = \sum_{i=D(G)}^k{k \choose i}\phi_k(G). 
\end{equation}



\begin{definition}
	Let $c$ be a coloring of a graph $G$. We say that $c$ supports an automorphism $f \in \aut(G)$ if $c(f(v)) = c(v)$ for all $v \in V(G)$. In other words, $f$ preserves the colors of vertices.
\end{definition}

For any coloring $c$ of a graph $G$ it can be easily observed that the stabilizer of $c$,  
$$S_c=\{f\in \aut(G):\;\; c \;\; {\mbox{supports}}\;\;f\}$$
 is a subgroup of $\aut(G)$ and the number of colorings equivalent to $c$ is $[\aut(G):\;S_c]$ (the index of $S_c$  in $\aut(G)$).
Therefore, $c$ is a distinguishing coloring of $G$ if and only if $S_c$ is the identity subgroup of $\aut(G)$.

For a graph $G$ we denote the number of distinguishing colorings of $G$ using at most $k$ colors by $\mathfrak{D}_k(G)$.

\begin{theorem}
	For any graph $G$ of order $n$, we have,
	\begin{equation}\label{d-phi}
	\mathfrak{D}_k(G)=\sum_{i=D(G)}^{n} {k \choose i} \phi_i(G)|\aut(G)|=\Phi_k(G)|\aut(G)|,	
	\end{equation}
	 
	which is a monic polynomial in $k$ of degree $n$.
\end{theorem}
\begin{proof}
	The equality \ref{d-phi} is clear by definitions. Note that in the summation above, the largest degree term versus $k$ is $k^n$ occurring in $k(k-1)\ldots(k-n+1)$, for $i=n$; and the previous terms have degrees less than $n$ in $k$. Hence  $\mathfrak{D}_k(G)$ is a monic polynomial of degree $n$. 
\end{proof}

\begin{definition}
	For a graph $G$ we call the polynomial $\mathfrak{D}_k(G)$ the distinguishing polynomial of $G$.
\end{definition}

Since the automorphisms of $G$ are the same as those of $G^c$,  the distinguishing polynomials of $G$ and $G^c$ are the same. Also, once we have one of the three functions 
$\phi_k(G)$, $\Phi_k(G)$ or $\mathfrak{D}_k(G)$, then the other two can be derived via equalities \ref{phi-phi} and \ref{d-phi}. 

\begin{example}\label{example6ta}
The following facts can easily be verified.
\begin{enumerate}
\item  $\mathfrak{D}_k(K_n) =k_{(n)}= k (k-1) \cdots (k-n+1)$. 
\item If $G$ is asymmetric of order $n$, then $\mathfrak{D}_k(G)=k^n$. Conversely, if  $\mathfrak{D}_k(G)=k^n$, then $G$ is  asymmetric  of order $n$.
\item If $S_n$ denotes the star with $n+1$ vertices, then $\mathfrak{D}_k(S_n) = k k_{(n)} = k^2 (k-1) \cdots (k-n+2).$ 
\item $\mathfrak{D}_k(2K_2) = \mathfrak{D}_k(C_4)  = k (k-1) (k-2) (k+1).$
\item If $G$ is a graph with no isolated vertex, then $\mathfrak{D}_k(G\cup K_1)=k\mathfrak{D}_k(G).$
\item  For the path $P_n$  on $n$ vertices, $\mathfrak{D}_k(P_n)=k^n-k^{\lceil \frac{n}{2} \rceil}$. Because every non-distinguishing coloring of $P_n$ is a coloring symmetric to the mid point of the path. 
\end{enumerate}

\end{example}

Note that item 6 in the example above  together with equation \ref{d-phi} and $|\aut(P_n)|=2$ easily gives $\Phi_k(P_n)=\frac{1}{2}(k^n-k^{\lceil \frac{n}{2}\rceil})$. This is done recursively in \cite{ahmadi2020number}, without giving an explicit formula.    

Another challenging case mentioned in \cite{ahmadi2020number} is computing  $\Phi_k$ for cycles. The following section is devoted to cycles. We consider disconnected graphs and complete multipartite graphs in section  \ref{disc}. At the end, in section \ref{conclusion}, we show that the multiplicity of zero in $\mathfrak{D}_k(G)$  is at least the number of orbits of $G$.

\section{Cycles}
As usual $C_n$ denotes the cycle of length $n \geq 3$. Recall that the automorphism group of $C_n$ is (isomorphic to) the dihedral group $D_{2n}$; consisting of all rotations and reflections of a regular $n$-gon . Therefore, a distinguishing coloring of $C_n$ can be regarded as one that fixes all rotations and reflections of the regular $n$-gon of unit edge. 
Let us make no difference between the regular $n$-gon of unit edge and $C_n$. 
We need also the degenerate cycles $C_1=K_1$ and $C_2=K_2$, whose automorphism groups are the identity and $\mathbb{Z}_2$, respectively.

We label the vertices of $C_n$ by $1,\ldots, n$ and  calculations are taken modulo $n$. For $1\leq i \leq n$, denote by $\rho_i$ the reflection relative to the diagonal passing through the vertex $i$. When $n$ is even, we have also reflections relative to the diagonals bisecting two opposite edges $\{i,i+1\}$  and $\{i+\frac{n}{2},i+1+\frac{n}{2}\}$. Such a reflection is denoted by $\rho_{i,i+1} =\rho_{i+\frac{n}{2},i+1+\frac{n}{2}}.$  Note that when $n$ is odd, $\rho_i$ is also the reflection relative to the diagonal bisecting the edge $\{i+\frac{n-1}{2}, i+\frac{n+1}{2}\}$. Despite this fact,
we restrict the reflections $\rho_{i,i+1}$ to the case $n$ even. See Figure \ref{fig:reflections}. 
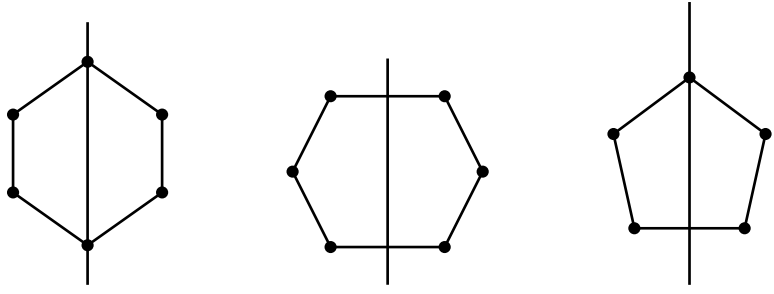
\begin{figure} [h]
	\centering
	\begin{center}
		\begin{tabular}{ c c c c c }
			\begin{tikzpicture}  [scale=0.35]
				\tikzstyle{every path}=[line width=1pt]
				
				\newdimen\ms
				\ms=0.1cm
				\tikzstyle{s1}=[color=black,fill,rectangle,inner sep=3]
				\tikzstyle{c1}=[color=black,fill,circle,inner sep={\ms/8},minimum size=1.6*\ms]
				\tikzstyle{c2}=[color=red,fill,circle,inner sep={\ms/8},minimum size=1.6*\ms]
				
				
				\coordinate (a1) at  (0,0);
				\coordinate(a2) at (-2.8,-2);
				\coordinate(a3) at (2.8,-2);
				\coordinate(a4) at (0,1.5);
				\coordinate(a5) at (0,-8.45);
				\coordinate(a6) at (-2.8,-4.95);
				\coordinate(a7) at (2.8,-4.95);
				\coordinate(a8) at (0,-6.95);				
				\draw [color=black] (a1) -- (a2);
				\draw [color=black] (a3) -- (a1);
				\draw [color=black] (a5) -- (a4);
				\draw [color=black] (a7) -- (a3);
				\draw [color=black] (a6) -- (a2);
				\draw [color=black] (a7) -- (a8);
				\draw [color=black] (a6) -- (a8);
				
				\draw (a1) coordinate[c1];
				\draw (a2) coordinate[c1];
				\draw (a3) coordinate[c1];
				\draw (a7) coordinate[c1];
				\draw (a6) coordinate[c1];
				\draw (a8) coordinate[c1];
			\end{tikzpicture}
			&
			\hspace{1cm}
			\begin{tikzpicture}  [scale=0.5]
				\tikzstyle{every path}=[line width=1pt]
				
				\newdimen\ms
				\ms=0.1cm
				\tikzstyle{s1}=[color=black,fill,rectangle,inner sep=3]
				\tikzstyle{c1}=[color=black,fill,circle,inner sep={\ms/8},minimum size=1.6*\ms]
				\tikzstyle{c2}=[color=red,fill,circle,inner sep={\ms/8},minimum size=1.6*\ms]
				
				
				\coordinate (a1) at  (0,1);
				\coordinate(a2) at (3,1);
				\coordinate(a3) at (0,-3);
				\coordinate(a4) at (3,-3);
				\coordinate(a5) at (4,-1);
				\coordinate(a6) at (-1,-1);
				\coordinate(a7) at (1.5,2);
				\coordinate(a8) at (1.5,-4);				
				\draw [color=black] (a1) -- (a2);
				\draw [color=black] (a6) -- (a1);
				\draw [color=black] (a5) -- (a2);
				\draw [color=black] (a5) -- (a4);
				\draw [color=black] (a4) -- (a3);
				\draw [color=black] (a3) -- (a6);
				\draw [color=black] (a7) -- (a8);
				
				\draw (a1) coordinate[c1];
				\draw (a2) coordinate[c1];
				\draw (a3) coordinate[c1];
				\draw (a4) coordinate[c1];
				\draw (a5) coordinate[c1];
				\draw (a6) coordinate[c1];
			\end{tikzpicture}
			&
			\hspace{1cm}
			\begin{tikzpicture}  [scale=0.5]
				\tikzstyle{every path}=[line width=1pt]
				
				\newdimen\ms
				\ms=0.1cm
				\tikzstyle{s1}=[color=black,fill,rectangle,inner sep=3]
				\tikzstyle{c1}=[color=black,fill,circle,inner sep={\ms/8},minimum size=1.6*\ms]
				\tikzstyle{c2}=[color=red,fill,circle,inner sep={\ms/8},minimum size=1.6*\ms]
				
				
				\coordinate (a1) at  (0,-0.5);
				\coordinate(a2) at (-2,-2);
				\coordinate(a3) at (2,-2);
				\coordinate(a4) at (0,1.5);
				\coordinate(a5) at (0,-6);
				\coordinate(a6) at (-1.45,-4.5);
				\coordinate(a7) at (1.45,-4.5);
				
				\draw [color=black] (a1) -- (a2);
				\draw [color=black] (a3) -- (a1);
				\draw [color=black] (a5) -- (a4);
				\draw [color=black] (a7) -- (a3);
				\draw [color=black] (a6) -- (a2);
				\draw [color=black] (a7) -- (a6);

				\draw (a1) coordinate[c1];
				\draw (a2) coordinate[c1];
				\draw (a3) coordinate[c1];
				\draw (a7) coordinate[c1];
				\draw (a6) coordinate[c1];
			\end{tikzpicture}
			\\ \vspace{1mm} \\
		\end{tabular}

	\end{center}
	\caption{Reflections are relative to diagonals}
	
	\label{fig:reflections}
\end{figure}
In order to compute $\mathfrak{D}_k(C_n)$, our first tool is the following lemma which is an straightforward consequence of the algebraic definition of $D_{2n}$.

\begin{lemma}\label{two different reflections}\hfill
	\begin{enumerate}
		\item[(a)] If a subgroup of $D_{2n}$ contains two reflections, then it contains a rotation.
		
		\item[(b)] If a subgroup of $D_{2n}$ contains a rotation and a reflection, then it contains more than one reflections.
	\end{enumerate}
\end{lemma}

Recall that a coloring $c$ of a graph $G$ supports an automorphism $f \in \aut(G)$ if $c(f(v))=c(v)$ for all $v\in V(G)$. In case of $C_n$, the following lemma gives  more insight on colorings supporting automorphisms.

\begin{lemma}\label{reflections and rotations of $C_n$}
	Let $c$ be a coloring of $C_n$. 
	\begin{enumerate}
		\item[(a)] If $c$ supports  two different reflections, then it supports a rotation.
		\item[(b)] If $c$ supports a reflection and a rotation, then it supports more than one reflections.
	
	\end{enumerate} 
\end{lemma}
\begin{proof}
	Note that the stabilizer $S_c$ of $c$ is a subgroup of $\aut(C_n)$ and use Lemma \ref{two different reflections}.
	\end{proof}

Let $d$ be a divisor of $n$ and consider a colored segment $[d]=\{1,2,\ldots d\}$. If we put this colored segment consecutively around $C_n$  ($\frac{n}{d}$ times,  clockwise), starting at an arbitrary vertex, then we have a coloring $c$ of $C_n$. In this case, we say that the coloring $c$ is generated by the colored segment $[d]$.
Note that when $d\neq n$, $c$ supports the rotation by the angle $2\pi \frac{d}{n}.$ On the other hand, if $c$ is a coloring supporting a rotation of $C_n$, then there exists a divisor $d\neq n$  of $n$, such that $c$ is generated by the colored segment $[d]$.

According to Lemma \ref{reflections and rotations of $C_n$}, the set of all non-distinguishing $k$-colorings of $C_n$ is partitioned into two disjoint subsets. One is consisting of those colorings supporting exactly one reflection, and the other is consisting of colorings  which support some rotation. We denote the former   by $\mathcal{M}_{n,k}$ and the latter by  $\mathcal{N}_{n,k}$.
%
Therefore, the distinguishing polynomial of $C_n$ would be  
\begin{equation} \label{dk}
\mathfrak{D}_k(C_n) = k^n - |\mathcal{M}_{n,k}| - |\mathcal{N}_{n,k}|.
\end{equation}
In the sequel, we count $|\mathcal{M}_{n,k}|$ and 
$|\mathcal{N}_{n,k}|$.

Note that any non-constant coloring $c$ of $C_n$ can be generated by more than one colored segments, because of cyclicity. Assume that the coloring $c$ is generated by a segment $[d]$, $d\neq 1$. For $i \in \{1,2, \ldots, n\}$, if the vertex $i$ does not match the first position of a copy of $[d]$ around the cycle, there exists another coloring of the segment $[d]$ generating $c$, in which the vertex $i$ matches the first position in a copy of $[d]$. In fact, the new colored segment is obtained from $[d]$ by moving the colors cyclically in the segment. For $d=1$, a trivial statement holds. This observation implies the following lemma.
	
	 \begin{lemma} \label{d'|d}
	 	Let $d$ be a divisor of $n$. The number of $k$-colorings of $C_n$ generated by colored segments of length $d$ is $k^d$. In addition, if $d'|d$, a $k$-coloring generated by a colored segment of length $d'$ is also generated by some colored segment of length $d$. 
	 \end{lemma}
	 \begin{proof}
	 	The number of $k$-colorings of the segment  $[d]$ is $k^d$, which we put it around $C_n$ consecutively from an arbitrary initial vertex. The initial vertex does not alter the counting, by the preceding observation. 
	 	The second assertion is clear. 
	 \end{proof}  
	
	\begin{lemma}\label{coloring-induced-on-cd}
		
		Suppose that $d$ is a divisor of $n$ and $[d]$ is a colored segment and $c$ is a coloring of $C_n$ generated by $[d]$. Then $c$ supports a reflection of $C_n$ if and only if any coloring of $C_d$ generated by $[d]$ supports a reflection of $C_d$. 
	
	\end{lemma}
	\begin{proof}	Throughout this proof,  without loss of generality, we assume that the vertex 1 matches the first position in a copy of $[d]$ around the corresponding cycle ($C_n$ or $C_d$). 
		
		First suppose that $c$ supports a reflection $\rho_i\in \aut(C_n),\;\; 1\leq i\leq n$. We can assume that $i=1$.  Hence, 
			\begin{equation} \label{cjvertex}
			c(j+1) = c(-j+1), \;\;\;  1 \leq j \leq n. 
		\end{equation}
		Here the calculations are taken modulo $n$. Since $d$ is a divisor of $n$, equation  
	\ref{cjvertex} holds modulo $d$, too. Consequently, coloring the vertices of $C_d$ according to the coloring of $[d]$ supports the reflection $\rho_1 \in \aut(C_d)$.  
		
		Second, suppose that $c$ supports the reflection $\rho_{i,i+1}$ which occurs for even $n$. We assume that $i=n$.  Consequently,
		 \begin{equation} \label{cj}
		c(j+1) = c(-j),\;\;\; 1\leq j\leq n.
		\end{equation}
		 Here the calculations are taken modulo $n$.  We color the vertices of $C_d$ according to the coloring of $[d]$. When $d$ is even, it supports $\rho_{d,1}\in \aut{(C_d})$, again by the equation \ref{cj}; since $d$ is a divisor of $n$ and the equation \ref{cj} also holds modulo $d$. If $d$ is odd, changing $j$ by $\frac{d+1}{2}+j$ in equation \ref{cj} leads to:
		 $$
		 c(\frac{d+1}{2}+j+1)=c(-\frac{d+1}{2}-j)=c(d-\frac{d+1}{2}-j)=c(\frac{d+1}{2}-j+1).
		 $$
		 This equation is nothing but the equation \ref{cjvertex} shifted by $\frac{d+1}{2}$, true also in modulo $d$. It shows that the coloring  supports the reflection $\rho_{\frac{d+1}{2}}=\rho_{\lceil \frac{d}{2}\rceil}\in \aut{(C_d)}.$

		Conversely, suppose that $c$ is a coloring of $C_d$ generated by a colored segment $[d]$, which supports a reflection $\rho \in \aut(C_d)$. If  $\rho = \rho_i$, $1\leq i\leq d$, again we can assume that $i=1$. Then, 
		$$c(j + 1) = c(-j + 1), \;\;\; 1 \leq j \leq d.$$ Here computations are taken modulo $d$. Consider the coloring $c'$ on $C_n$ generated by the colored segment $[d]$. So, for each $1 \leq p  \leq n$, if $p=qd+j$, for  integers $q$ and $1 \leq j \leq d$, then $c'(p)=c(j)$.  That is, if two vertices of $C_n$ are congruent modulo $d$,  they are assigned the same color as in $[d]$.  Hence,  
		$$c'(p+1)=c(j+1)=c(-j+1)=c'(-p+1).$$   
It follows that $c'$ supports $\rho_1\in \aut(C_n)$.
	(Whenever $n$ is even and $d$ is odd, similar corresponding calculations show that $c'$ also supports the reflection $\rho_{\frac{d+1}{2},\frac{d+3}{2}} \in \aut(C_n)$)
	
		Finally, let $\rho = \rho_{i,i+1}\in \aut(C_d)$,  which occurs for $d$ (and $n$) even. We assume that $i=d$. The proof is similar and runs as in the previous case. However, for the sake of completeness and to avoid confusions,  the proof is brought. For each $1\leq j \leq d$, we have $c(j+1)=c(-j)$, where the calculations are taken modulo $d$.   
		Consider the coloring $c'$ on $C_n$ generated by the colored segment $[d]$. So, for each $1 \leq p  \leq n$, if $p=qd+j$, for integers $q$ and $1 \leq j \leq d$, then $c'(p)=c(j)=c(-j+1)=c'(-p+1)$. Hence,  
		$$c'(p+1)=c(j+1)=c(-j)=c'(-p), \;\;\; 1\leq p\leq n.$$
		This shows that $c'$ supports $\rho_{n,1}\in \aut(C_n)$. 
\end{proof}

\begin{corollary}\label{counting} Let $d|n$ and $1\leq i\leq n$. Then,	\begin{itemize}
	 
	\item [(a) ]	The number of $k$-colorings of $C_n$ which are generated by a colored segment of length $d$ and support the  reflection $\rho_i$ is
  $ k^{\lceil \frac{d + 1}{2}\rceil}$. 
	\item [(b)] If $n$ is even, the number of $k$-colorings of $C_n$ which are generated by a colored segment of length $d$ and support the  reflection $\rho_{i,i+1}$ is 
	  $ k^{\lfloor \frac{d + 1}{2}\rfloor}$.  
\end{itemize}
	\begin{proof}
		We employ Lemma \ref{coloring-induced-on-cd} and  its proof.\\
	(a)	 Without loss of generality, let $i=1$. Such a coloring on $C_n$ induces a coloring on $C_d$ supporting the reflection $\rho_1$.  If $d$ is odd, this counts as $ k^{\frac{d + 1}{2}}$ colorings on $C_d$; since the vertex 1 has an arbitrary color and half of the remaining vertices determine the coloring. If $d$ is even, it counts as  $ k^{\frac{d }{2}+1}$ colorings on $C_d$; since in addition to vertex 1, its opposite vertex, $\frac{d}{2}+1$, has an arbitrary color, too; and half of the remaining vertices determine the coloring.  Both numbers  are equal to
	$ k^{\lceil \frac{d + 1}{2}\rceil}.$
	
	(b) Without loss of generality let $i=n$. 
	 First suppose that $d$ is odd. Since the coloring of $[d]$ around $C_n$ supports $\rho_{n,1}$, it supports the reflection $\rho_{\lceil \frac{d}{2}\rceil}$  on $C_d$. The number of such colorings is $ k^{\frac{d + 1}{2}},$  as computed in (a).
	
	If $d$ is even, in order to support $\rho_{n,1}$ on $C_n$, the segment $[d]$  is colored symmetrically with respect to its mid point. The number of such colorings of $[d]$ is 
	$ k^{\frac{d}{2}}$.
	Here both numbers are equal to $ k^{\lfloor \frac{d + 1}{2}\rfloor}$.
	\end{proof}
	
\end{corollary}

Let $n = p_1^{n_1}p_2^{n_2}\cdots p_t^{n_t}$ be the canonical decomposition of $n$ to  prime factors. For each subset $\mathcal{A}$ of $W=\{1, 2, \ldots, t\}$, define the integer 
$$n_{\mathcal{A}} = (\prod_{i \in W \setminus \mathcal{A}}p_i^{n_i}) (\prod_{j \in \mathcal{A}}p_j^{n_j-1}),$$ which is a divisor of $n$.  When $\mathcal{A}=\emptyset$,  $n_{\mathcal{A}}=n$; and  $n_{\mathcal{A}}$ is a maximal divisor of $n$ if and only if $|\mathcal{A}|=1.$  Also, for any 
subset $\mathcal{A}$ of $W$, 
we have $$n_{\mathcal{A}} = \gcd(n_{\{x\}}~:~ x \in \mathcal{A}).$$




\begin{lemma}\label{the-number-of-support-rotation}
	Let $n = p_1^{n_1}p_2^{n_2}\cdots p_t^{n_t}$ be the canonical decomposition of $n$ to prime factors. Then, 
	\[ |\mathcal{N}_{n,k}| =k^n - \sum_{\mathcal{A} \subseteq \{1,2, \ldots, t\}}(-1)^{|\mathcal{A}|} k^{n_{\mathcal{A}}}.\]
\end{lemma}
\begin{proof}
	Let $S$ be the set of all $k$-colorings of $C_n$ and let $S_i$ be the subset of $S$ consisting  colorings generated by  colored segments of length $n_{\{i\}}$.  Every coloring of $C_n$ which supports a rotation is generated by a colored segment $[d]$ for some $d|n$, $d\neq n$, and such a $d$ is a divisor of   $n_{\{i\}}$, for some $i\in \{1,2,\ldots t\}$. By Lemma \ref{d'|d},
	$\mathcal{N}_{n,k}=\bigcup_{i=1}^{t}S_i$.
	Also, for each $\emptyset \neq \mathcal{A} \subseteq \{1,2, \ldots, t\}$, the set $\bigcap_{i\in \mathcal{A}}{S_i}$ is the set of colorings in $S$ which are generated by some colored segment of length $n_\mathcal{A}$. Therefore, by Lemma \ref{d'|d},
	$|\bigcap_{i\in \mathcal{A}}{S_i}|=
	k^{n_{\mathcal{A}}}.$ 
By the principle of inclusion-exclusion, 
	\[ |\mathcal{N}_{n,k}| =\sum_{\emptyset \neq \mathcal{A} \subseteq \{1,2, \ldots, t\}}(-1)^{|\mathcal{A}|+1} k^{n_{\mathcal{A}}}.\]
	Finally, taking into account $\mathcal{A}=\emptyset$ and noting that $n_\emptyset=n$ results in:
		\[ |\mathcal{N}_{n,k}| =k^n -\sum_{ \mathcal{A} \subseteq \{1,2, \ldots, t\}}(-1)^{|\mathcal{A}|} k^{n_{\mathcal{A}}}.\]
\end{proof}



\begin{lemma}\label{number of automorphism supporting only one reflection}
	Let $n = p_1^{n_1}p_2^{n_2}\cdots p_t^{n_t}$ be the canonical decomposition of $n$ to prime factors. Then 	
	
	\[|\mathcal{M}_{n,k}| = \frac{n}{2}\sum_{\mathcal{A} \subseteq \{1,2, \ldots, t\}}(-1)^{|\mathcal{A}|}( k^{\lceil\frac{n_{\mathcal{A}}+1}{2}\rceil} + k^{\lfloor\frac{n_{\mathcal{A}}+1}{2}\rfloor}).\]

\end{lemma}

\begin{proof}
First, suppose  $X$ is the set of all $k$-colorings of $C_n$ supporting the reflection $\rho_1 \in \aut(C_n)$ and  $X_j$, $1 \leq j \leq n$, is the subset of $X$ consisting of $k$-colorings generated by the segment of length $n_{\{j\}}$. By Corollary \ref{counting} (a), we have $|X|=k^{\lceil \frac{n+1}{2}\rceil}.$ 
Note that a coloring $c$ in $X$ supports no reflection of $C_n$ other than $\rho_1$, if and only if $c \notin X_j$, for all $j$ (otherwise, contradicts Lemma \ref{reflections and rotations of $C_n$}). 
	 For $\emptyset \neq \mathcal{A} \subseteq \{1,2,\ldots, t\}$, the set $\bigcap_{j\in \mathcal{A}}{X_j}$ is the subset of $X$ consisting of  colorings generated by the segment of length $n_\mathcal{A}$. They also support  $\rho_1 \in \aut(C_n)$. Therefore, Corollary \ref{counting} (a) implies that 
	  $	|\bigcap_{j\in \mathcal{A}}{X_j}|
	  = k^{\lceil\frac{n_{\mathcal{A}}+1}{2}\rceil}$. 

We can now find the number of $k$-colorings of $C_n$ which support only $\rho_1$ and no more reflections (and of course no rotation, by Lemma \ref{reflections and rotations of $C_n$}. By the inclusion-exclusion principle,
$$|X\setminus(X_1\cup X_2\cup \ldots\cup  X_t)|=k^{\lceil\frac{n+1}{2}\rceil}-\sum_{\emptyset \neq \mathcal{A} \subseteq \{1,2, \ldots, t\}}(-1)^{|\mathcal{A}|+1} k^{\lceil\frac{n_{\mathcal{A}}+1}{2}\rceil}.$$
If we take into account $\mathcal{A}=\emptyset$ and note that $n_\emptyset=n$, the equality above becomes, 
$$|X\setminus(X_1\cup X_2\cup \ldots\cup  X_t)|=\sum_{ \mathcal{A} \subseteq \{1,2, \ldots, t\}}(-1)^{|\mathcal{A}|} k^{\lceil \frac{n_{\mathcal{A}}+1}{2}\rceil}.$$
	
	For each $1 \leq i \leq n$ we can use the same argument as $\rho_1$ for the reflection $\rho_i \in \aut(C_n)$. If $n$ is odd, there are no reflections of the form $\rho_{i,i+1}$. Since in this case all $n_{\mathcal{A}}$'s are odd we have the desired equality: 
	\[|\mathcal{M}_{n,k}|= 
	n \sum_{\mathcal{A} \subseteq \{1,2, \ldots, t\}}(-1)^{|\mathcal{A}|} k^{\lceil\frac{n_{\mathcal{A}}+1}{2}\rceil}=	 \frac{n}{2}\sum_{\mathcal{A} \subseteq \{1,2, \ldots, t\}}(-1)^{|\mathcal{A}|}( k^{\lceil\frac{n_{\mathcal{A}}+1}{2}\rceil} + k^{\lfloor\frac{n_{\mathcal{A}}+1}{2}\rfloor}).\]
	 
	 Now suppose that $n$ is even. Since $\rho_i=\rho_{\frac{n}{2}+i}$, the number of $k$-colorings of $C_n$ which support some $\rho_i$ and no more reflections is: 
	 	\[ \frac{n}{2} \sum_{\mathcal{A} \subseteq \{1,2, \ldots, t\}}(-1)^{|\mathcal{A}|} k^{\lceil\frac{n_{\mathcal{A}}+1}{2}\rceil}.\]

  In addition, a similar arguments and Lemma \ref{counting} (b) reveal that the number of $k$-colorings of $C_n$ which only support $\rho_{n,1}$ and no more reflections is equal to:
   	\[  \sum_{\mathcal{A} \subseteq \{1,2, \ldots, t\}}(-1)^{|\mathcal{A}|} k^{\lfloor\frac{n_{\mathcal{A}}+1}{2}\rfloor}.\]
   	Since $\rho_{i,i+1}=\rho_{\frac{n}{2}+i,\frac{n}{2}+i+1}$, the number of $k$-colorings of $C_n$ which support some $\rho_{i,i+1}$ and no more reflections is: 
   	\[ \frac{n}{2} \sum_{\mathcal{A} \subseteq \{1,2, \ldots, t\}}(-1)^{|\mathcal{A}|} k^{\lfloor\frac{n_{\mathcal{A}}+1}{2}\rfloor}.\] 
   	Therefore, 
   		\[|\mathcal{M}_{n,k}| = \frac{n}{2}\sum_{\mathcal{A} \subseteq \{1,2, \ldots, t\}}(-1)^{|\mathcal{A}|}( k^{\lceil\frac{n_{\mathcal{A}}+1}{2}\rceil} + k^{\lfloor\frac{n_{\mathcal{A}}+1}{2}\rfloor}).\]	
   		
   	The proof is now complete.
\end{proof}
Now all ingredients are ready to compute the distinguishing polynomials of cycles. 
\begin{theorem}
	Let $n = p_1^{n_1}p_2^{n_2}\cdots p_t^{n_t}$ be the canonical decomposition of $n$ to prime factors. Then,
	\begin{itemize}
		\item[(i)] 
		If $n$ is odd,
		\[\mathfrak{D}_k(C_n) =\sum_{\mathcal{A} \subseteq \{1,2, \ldots, t\}}(-1)^{|\mathcal{A}|}(k^{n_{\mathcal{A}}}- n k^{\frac{n_{\mathcal{A}}+1}{2}}).\]
		\item
		[(ii)] 
		If $n$ is even,
		\[\mathfrak{D}_k(C_n) =\sum_{\mathcal{A} \subseteq \{1,2, \ldots, t\}}(-1)^{|\mathcal{A}|}(k^{n_{\mathcal{A}}} - \frac{n}{2} k^{\lceil\frac{n_{\mathcal{A}}+1}{2}\rceil} - \frac{n}{2} k^{\lfloor\frac{n_{\mathcal{A}}+1}{2}\rfloor}).\]
	\end{itemize}
\end{theorem}
\begin{proof}
Use Equation \ref{dk}, Lemma \ref{the-number-of-support-rotation} and Lemma \ref{number of automorphism supporting only one reflection}.
\end{proof}
	




\begin{corollary}	\label{prime p} For any odd prime $p$,
	\begin{itemize}
	\item 
	$ \mathfrak{D}_k(C_p) = k^p - pk^{\frac{p+1}{2}} + (p-1)k.$
 \item $ \mathfrak{D}_k(C_{p^2}) = k^{p^2} - p^2k^{\frac{p^2+1}{2}} -k^p+p^2k^{\frac{p+1}{2}}.$
	\end{itemize}
\end{corollary} 

\begin{example} \label{c4c6}\hfill \begin{itemize}
		\item 
	$\mathfrak{D}_k(C_4)=k^4-2k^3-k^2+2k$ (given in split form in Example \ref{example6ta}  )
	\item
 $\mathfrak{D}_k(C_{6}) =k^6-3k^4-4k^3+8k^2-2k.$
	\end{itemize} 
\end{example}



\section{Disconnected graphs, Complete Multipartite Graphs And Joins }\label{disc}
In contrast to the chromatic polynomial, the distinguishing polynomial of a disjoint union of graphs is not necessarily equal to the product of their individual polynomials, unless some obvious necessary  conditions occur.  However, we can compute the distinguishing polynomial of a disjoint union of graphs versus the distinguishing polynomials of the components. For a graph $G$,  denote the disjoint union of $m$ copies of $G$ by $mG$. 

\begin{lemma}\label{Dist-Polynomial-Self-Union}
	For every connected graph $G$: 
	\[\Phi_k(mG) =\frac{1}{m!}\Phi_k(G)_{(m)}=\frac{1}{m!} \prod_{i=0}^{m-1}( \Phi_k(G) - i).\] 
\end{lemma}
\begin{proof}
	Show the copies of $G$ by $G_1,G_2,\cdots ,G_m$. There are $\Phi_k(G)$ nonequivalent distinguishing $k$-colorings of $G_1$. Let $c$ be a distinguishing $k$-coloring of $G_1$. To avoid mapping of $G_1$ and $G_2$ onto each other, we must have a coloring of $G_2$ nonequivalent to $c$. So, there exist $\Phi_k(G)-1$ nonequivalent distinguishing colorings for $G_2$. Similarly, we must avoid mappings of $G_1$, $G_2$ and $G_3$ onto each other. Hence, there exist $\Phi_k(G)-2$ nonequivalent distinguishing $k$-colorings for $G_3$; and so on. Therefore, there are
	$\prod_{i=0}^{m-1}( \Phi_k(G) - i)=\Phi_k(G)_{(m)}$  nonequivalent distinguishing $k$-colorings for $mG$, if the copies of $G$ are considered different objects. But the copies are the same objects, so we have to divide	$\prod_{i=0}^{m-1}( \Phi_k(G) - i)=\Phi_k(G)_{(m)}$ by $m!$ to find  $\Phi_k(mG).$
\end{proof}

In order to compute $\mathfrak{D}_k(mG)$, we also need the following known lemma, which can easily be verified.

\begin{lemma}\label{autmg}
	For any connected graph $G$, $|\aut(mG)|=m!|\aut(G)|^m$.	\hfill $\Box   $
\end{lemma}

The previous two lemmas and \ref{d-phi} gives $\mathfrak{D}_k(mG)$, whenever $G$ is connected.
\begin{corollary}\label{dkmg}
	For a connected graph $G$, $$\mathfrak{D}_k(mG)=\prod_{i=0}^{m-1}(\mathfrak{D}_k(G)-i|\aut(G)|) $$
	
\end{corollary}
\begin{proof}
	By   Lemma \ref{autmg} and Corollary \ref{dkmg}, $$\mathfrak{D}_k(mG)=\Phi_k(mG)m!|\aut(G)|^m=\frac{1}{m!}\Phi_k(G)_{(m)}m!|\aut(G)|^m=\prod_{i=0}^{m-1}(\mathfrak{D}_k(G)-i|\aut(G)|).$$
\end{proof}

Now, gathering the results above, everything is ready to compute the distinguishing polynomial of a disjoint union versus distinguishing polynomials of its components.

\begin{theorem} \label{disjoin union}
	Let $G_1,G_2,\cdots,G_t$ be mutually non-isomorphic connected graphs and   $m_1$, $m_2$, $\cdots,m_t$ be positive integers. Then 
	\begin{equation}
	\mathfrak{D}_k(\bigcup_{i=1}^{t}m_iG_i)=\prod_{i=1}^{t}\mathfrak{D}_k(m_iG_i)
	=\prod_{i=1}^{t}\prod_{j=0}^{m_i-1}(\mathfrak{D}_k(G_i)-j|\aut(G_i)|).
	\end{equation}
\end{theorem}
\begin{proof}
	The proof is plain. For $i\neq j$, no automorphism can map a copy of $G_i$ to a copy of $G_j$.
	In other words, any automorphism of $\bigcup_{i=1}^{t}m_iG_i$ maps $m_iG_i$, $1\leq i\leq t$, onto itself, hence the left equality holds.  Write the right equality  by Corollary \ref{dkmg}.  
\end{proof}
As a corollary,  distinguishing polynomials of complete multipartite graphs are derived.

\begin{theorem}
Suppose that $n_i$, $1\leq i\leq t$ are distinct positive integers and  $H$ is the complete multipartite graph with $m_i$ parts of order $n_i$, $1\leq i\leq t$. Then 
$$\mathfrak{D}_k(H)=\prod_{i=1}^{t}\prod_{j=0}^{m_i-1}
(k_{(n_i)}-jn_i!)$$
\end{theorem}
\begin{proof}
As we noted earlier,  the distinguishing polynomials of a graph and its complement are the same. The complement of the complete multipartite $H$ is the disjoint union $\bigcup_{i=1}^{t}m_iK_{n_i}$
and $|\aut(K_{n_i})|=n_i!$. Now Theorem \ref{disjoin union} gives the result. 
\end{proof}
\begin{example}\hfill
\begin{itemize}
	\item $\mathfrak{D}_k(K_{2,3})=\mathfrak{D}_k(K_2\cup K_3)=k^2(k-1)^2(k-2) =k^5-4k^4+5k^3-2k^2$.
	\item $\mathfrak{D}_k(K_{3,3})=\mathfrak{D}_k(2K_3)=k(k-1)(k-2)[k(k-1)(k-2)-6]=
	k(k-1)(k-2)(k-3)(k^2+2)=k^6-6k^5+13k^4-18k^3+22k^2-12k$.
	\item$\mathfrak{D}_k(K_{2,2,2})=\mathfrak{D}_k(3K_2)=k(k-1)[k(k-1)-2][k(k-1)-4]=
	k(k-1)(k-2)^2(k+1)(k^2-k-4)=k^6-3k^5-3k^4+11k^3+2k^2-8k$.
	\end{itemize}
\end{example} 

Finally, if we consider the join $G \vee H$ of two graphs $G$ and $H$, we have 
$\mathfrak{D}_k(G \vee H)=\mathfrak{D}_k(G^c \cup H^c)$, which can be computed by Theorem \ref{disjoin union}. For example, the distinguishing polynomial of a wheel is at hand, since a wheel is the join of a cycle with a single vertex. So is for a fan, the join of a path and a single vertex.

\section{  Conclusion  }\label{conclusion}
Like all over mathematics, the notion distinguishing polynomial brings up many new problems and aspects. After paths and cycles, we are  interested in knowing distinguishing polynomials of trees, hypercubes and Kneser graphs, ..., specially the Petersen graph. On the other hand, the theoretical aspects of this polynomial are also important. In contrast to the chromatic polynomial, we do not know exactly how the coefficients of $\mathfrak{D}_k(G)$ are related to the graph structure. 

In particular, the multiplicity of zero  as a root of $\mathfrak{D}_k(G)$ seems to be an important invariant of the graph $G$.     
Roughly speaking, any increase in  multiplicity of zero results in less symmetry and vice versa. For the moment, we have no exact interpretation of such multiplicity versus the graph structure. However, we show that the  multiplicity of zero in the distinguishing polynomial of $G$ is at least the number of its orbits. Recall that an orbit $\mathcal{O}$ of a graph $G$ is a maximal subset of vertices of $G$ for which if $x,y\in \mathcal{O}$, then there exists $f\in \aut(G)$  such that $f(x)=y$. The vertex set of $G$ is partitioned into its orbits.

\begin{definition}
	We call two distinguishing colorings  $c$ and $c'$ of $G$ similar if they induce the same partition on each orbit of $G$. That is, if $\mathcal{O}$ is any orbit of $G$ and $x,y\in \mathcal{O}$, then $c(x)=c(y)$ if and only if $c'(x)=c'(y)$.  
\end{definition}
\begin{example} 
	Any two distinguishing colorings of the path $P_3$ as well as complete graph $K_n$ and any asymmetric graph are similar. On the other hand, consider the path $P_4$ with vertices $1,2,3,4$.  Let $\alpha$, $\beta$, $\gamma$ and $\delta$ be different colors. The following three distinguishing colorings $c$, $c'$ and $c''$ are mutually non-similar: 
	$c(1,2,3,4)=(\alpha,\beta,\gamma,\delta)$, $c'(1,2,3,4)=(\alpha,\alpha,\alpha,\gamma)$ and 
	$c''(1,2,3,4)=(\alpha,\alpha,\beta,\alpha)$. The coloring $c'''(1,2,3,4)=(\alpha,\beta,\alpha,\beta)$ is similar to $c$. In fact, each distinguishing coloring of $P_4$ is  similar  to one of the colorings $c, c', c''$.

	\end{example}

	Note that similarity is an equivalence relation on the set of all distinguishing $k$-colorings of $G$.  Let $c$ be a distinguishing coloring of $G$. The number of colorings similar to $c$ can be computed as follows. If $\mathcal{O}$ is an orbit of $G$ partitioned as $X_1,X_2,\cdots, X_t$ by $c$,  the vertices in $\mathcal{O}$ can be colored in
	\begin{equation}\label{in an orbit}
		k(k-1)\cdots (k-t+1)
	\end{equation} 
	number of ways; and the number of distinguishing $k$-colorings similar to $c$ is equal to the product of products of the form \ref{in an orbit} for all orbits. For each orbit we have a factor $k$, hence there are $k^q g_c(k)$  distinguishing $k$-coloring of $G$ similar to $c$, where  $g_c(k)$ is a suitable nonzero polynomial. The set of all distinguishing $k$-colorings of $G$  is partitioned into equivalence classes of similarity, and the cardinality of each class is divisible by $k^q$. Therefore, $\mathfrak{D}_k(G)$ is divisible by $k^q$. Aggregating, we have proved the following theorem. 
	
	\begin{theorem}
		For any graph $G$ with $q$ orbits, the multiplicity of zero as a root of $\mathfrak{D}_k(G)$ is at least $q$. \hfill $\Box$
	\end{theorem}   
	
	For most of the graphs studied in previous sections, namely paths, complete graphs, complete multipartite graphs, asymmetric graphs  and cycles of odd prime lengths, the multiplicity of zero as a root of  $\mathfrak{D}_k(G)$ is exactly the number of their orbits. However,  as corollary \ref{prime p} shows, for an odd prime $p$, while the cycle  $C_{p^2}$ is vertex transitive and has only one orbit, the multiplicity of zero in  $\mathfrak{D}_k(C_{p^2})$ exceeds one. It also shows that there is no limit for the multiplicity of zero in $\mathfrak{D}_k(G)$ even when $G$ is vertex transitive. Other roots are also interesting, possible negative integers, real or complex roots.

	Finally,  as stated, knowing the automorphism groups and orbits play important roles in this theory; nevertheless, they are not sufficient to determine the distinguishing polynomials, as the graphs in previous  sections demonstrate. In addition to employing the automorphisms and orbits, different techniques were used to determine the distinguishing polynomials of illustrated graphs. These observations suggest that the complexity of determining $\mathfrak{D}_k(G)$ is NP, although it remains to be seen whether this statement is true.





					\bibliographystyle{plain}

			\end{document}